\documentclass[12pt,a4paper]{article}

\usepackage{geometry}
\usepackage[T1]{fontenc}
\usepackage[utf8]{inputenc}
\usepackage{authblk}
\usepackage{color}

\usepackage{graphics}
\usepackage{tikz,pgf}
\usetikzlibrary{decorations.markings, decorations.pathmorphing}
\usetikzlibrary{arrows,decorations,backgrounds,scopes,plotmarks}
\usepackage{subfigure}

\usepackage[english]{babel}
\usepackage[all]{xy}
\usepackage{marvosym}
\usepackage{stmaryrd,mathrsfs,amsmath,amssymb,bm,amsfonts}
\usepackage{enumerate}
\usepackage{amsmath,amsthm,amssymb,amscd}

\newtheorem{thm}{Theorem}
\newtheorem{defn}[thm]{Definition}

\newtheorem{lem}[thm]{Lemma}

\numberwithin{equation}{subsection}

\begin{document}

\title{Planar order on vertex poset}

\author[a,b]{Xuexing Lu}

\affil[a]{\small School of Mathematical Sciences, University of Science and Technology of China, Hefei, China}
\affil[b]{Wu Wen-Tsun Key Laboratory of Mathematics, Chinese Academy of Sciences, Hefei, China}

\renewcommand\Authands{ and }
\maketitle

\begin{abstract}
A planar order is a special linear extension of the edge poset (partially ordered set) of a processive plane graph. The definition of a planar order makes sense for any finite poset and is equivalent to the one of a conjugate order. Here it was proved that there is a planar order on the vertex poset of a processive planar graph naturally induced from the planar order of its edge poset.
\end{abstract}

\text{\textit{Keywords}: edge poset, vertex poset, planar order}\\

\section{Introduction}
The notion of a processive plane graph, a special case of Joyal and Street's progressive plane graph \cite{[JS91]},  was introduce in \cite{[HLY16]} as  a graphical tool for tensor calculus in semi-groupal categories. In \cite{[HLY16]}, we gave a totally combinatorial characterization of an equivalence class of processive graphs in terms of the notions of a \textbf{POP-graph} which is a \textbf{processive graph} (a special kind of acyclic directed graph) equipped with a \textbf{planar order} (a special linear order of the edges).

However, it turns out that the notion of a planar order can be defined for a general finite poset (partially ordered set) and essentially equivalent to the one of a conjugate order \cite{[FM96]}, which is an important notion in the study of planar posets. So this raises an interesting question: for a processive graph, are there some relations between planar orders on its edges and planar orders on its vertices? In this paper, we will give a positive answer to this question by showing that any planar order of edges of a processive graph naturally induces a planar order of vertices.
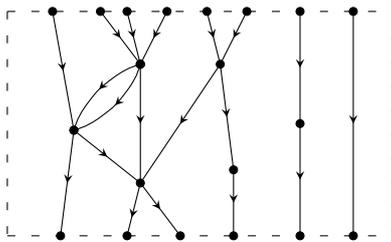
\section{Processive plane graph}
\begin{defn}
A \textbf{processive plane graph} is an acyclic directed graph drawn in a plane box with the properties that: $(1)$ all edges monotonically decrease in the vertical direction; $(2)$ all sources and sinks are of degree one; and $(3)$ all sources and sinks are placed on the horizontal boundaries of the plane box.
\end{defn}
Figure $1$ shows an example.

\begin{center}
\begin{tikzpicture}[scale=0.35]
\node (v2) at (-4,3) {};
\draw[fill] (-1.5,5.5) circle [radius=0.15];
\node (v1) at (-1.5,5.5) {};
\node (v7) at (-1.5,1) {};
\node (v9) at (1.5,5.5) {};
\node (v14) at (2,1.5) {};
\node (v3) at (-3,7.5) {};
\node (v4) at (-2,7.5) {};
\node (v5) at (-0.5,7.5) {};
\node (v6) at (-4.8,7.5) {};
\node (v11) at (-4.5,-1) {};
\node (v12) at (-2,-1) {};
\node (v13) at (0,-1) {};
\node (v15) at (2,-1) {};
\node (v8) at (1,7.5) {};
\node (v10) at (2.5,7.5) {};
\node  at (-2.5,3.5) {};
\node  at (-3,5.2) {};
\node  at (-1.2,3.3) {};
\node  at (0.5,3.25) {};
\node  at (2.2,3.7) {};
\node  at (-3,1.7) {};
\draw[fill] (-4,3) circle [radius=0.15];
\draw[fill] (v1) circle [radius=0.15];
\draw[fill] (v7) circle [radius=0.15];
\draw[fill] (v9) circle [radius=0.15];
\draw[fill] (v14) circle [radius=0.15];
\draw[fill] (v1) circle [radius=0.15];
\draw[fill] (v2) circle [radius=0.15];
\draw[fill] (v3) circle [radius=0.15];
\draw[fill] (v4) circle [radius=0.15];
\draw[fill] (v5) circle [radius=0.15];
\draw[fill] (v6) circle [radius=0.15];
\draw[fill] (v8) circle [radius=0.15];
\draw[fill] (v10) circle [radius=0.15];
\draw[fill] (v11) circle [radius=0.15];
\draw[fill] (v12) circle [radius=0.15];
\draw[fill] (v13) circle [radius=0.15];
\draw[fill] (v15) circle [radius=0.15];

\draw  plot[smooth, tension=1] coordinates {(v1) (-2.5,5)  (-3.5,4) (v2)}[postaction={decorate, decoration={markings,mark=at position .5 with {\arrow[black]{stealth}}}}];
\draw  plot[smooth, tension=1] coordinates {(v1) (-2,4.5)  (-3,3.5) (v2)}[postaction={decorate, decoration={markings,mark=at position .5 with {\arrow[black]{stealth}}}}];

\draw  (-3,7.5) -- (-1.5,5.5)[postaction={decorate, decoration={markings,mark=at position .5 with {\arrow[black]{stealth}}}}];
\draw  (-2,7.5) -- (-1.5,5.5)[postaction={decorate, decoration={markings,mark=at position .5 with {\arrow[black]{stealth}}}}];
\draw  (-0.5,7.5) -- (-1.5,5.5)[postaction={decorate, decoration={markings,mark=at position .5 with {\arrow[black]{stealth}}}}];

\draw  (-4.8,7.5)-- (-4,3)[postaction={decorate, decoration={markings,mark=at position .5 with {\arrow[black]{stealth}}}}];
\draw  (-1.5,5.5)  -- (-1.5,1)[postaction={decorate, decoration={markings,mark=at position .5 with {\arrow[black]{stealth}}}}];
\draw  (-4,3) -- (-1.5,1)[postaction={decorate, decoration={markings,mark=at position .5 with {\arrow[black]{stealth}}}}];

\draw (1,7.5)--(1.5,5.5)[postaction={decorate, decoration={markings,mark=at position .5 with {\arrow[black]{stealth}}}}];
\draw  (2.5,7.5) -- (1.5,5.5)[postaction={decorate, decoration={markings,mark=at position .5 with {\arrow[black]{stealth}}}}];
\draw  (1.5,5.5) -- (-1.5,1)[postaction={decorate, decoration={markings,mark=at position .5 with {\arrow[black]{stealth}}}}];
\draw  (-4,3) -- (-4.5,-1)[postaction={decorate, decoration={markings,mark=at position .5 with {\arrow[black]{stealth}}}}];
\draw  (-1.5,1) -- (-2,-1)[postaction={decorate, decoration={markings,mark=at position .65 with {\arrow[black]{stealth}}}}];
\draw  (0,-1) -- (-1.5,1)[postaction={decorate, decoration={markings,mark=at position .5 with {\arrowreversed[black]{stealth}}}}];
\draw  (1.5,5.5) -- (2,1.5)[postaction={decorate, decoration={markings,mark=at position .5 with {\arrow[black]{stealth}}}}];
\draw  (2,1.5) -- (2,-1)[postaction={decorate, decoration={markings,mark=at position .5 with {\arrow[black]{stealth}}}}];

\node (v17) at (6.5,7.5) {};
\node (v16) at (6.5,-1) {};
\draw[fill] (v16) circle [radius=0.15];
\draw[fill] (v17) circle [radius=0.15];
\draw  (6.5,-1) -- (6.5,7.5)[postaction={decorate, decoration={markings,mark=at position .5 with {\arrowreversed[black]{stealth}}}}];
\node (v18) at (4.5,7.5) {};
\node (v19) at (4.5,-1) {};
\node (v20) at (4.5,3.25) {};
\draw[fill] (v18) circle [radius=0.15];
\draw[fill] (v19) circle [radius=0.15];

\draw[fill] (v20) circle [radius=0.15];
\draw  (4.5,-1) -- (4.5,3.25)[postaction={decorate, decoration={markings,mark=at position .5 with {\arrowreversed[black]{stealth}}}}];
\draw  (4.5,3.24) -- (4.5,7.5)[postaction={decorate, decoration={markings,mark=at position .5 with {\arrowreversed[black]{stealth}}}}];

\draw [loosely dashed] (-6.5,7.5)--(-6.5,-1);
\draw [loosely dashed] (8,7.5)--(8,-1);
\draw [loosely dashed] (-6.5,7.5)--(8,7.5);
\draw [loosely dashed] (-6.5,-1)--(8,-1);

\end{tikzpicture}

Figure $1$. A processive plane graph
\end{center}

Processive plane graphs can also be defined in terms of \textbf{processive graphs} \cite{[HLY16]} and their \textbf{boxed planar drawings}.
\begin{defn}
A processive graph is an acyclic directed graph with all its sinks and sources of degree one.
\end{defn}
A planar drawing of processive graph  $G$ is called \textbf{boxed} \cite{[JS91]} if  $G$ is drawn in a plane box with all sinks of $G$  on one horizontal boundary of the plane box and all sources of $G$  on the other horizontal boundary of the plane box. A planar drawing of an acyclic directed graph is called \textbf{upward} if all edges increases monotonically in the vertical direction (or other fixed direction). Thus a processive plane graph is exactly a boxed and upward planar drawing of a processive graph.
\begin{defn}
Two processive plane graphs are \textbf{equivalent} if they are connected by a planar isotopy such that each intermediate planar drawing is boxed (not necessarily upward).
\end{defn}
In \cite{[HLY16]}, equivalence classes of processive plane graphs are mainly used to construct a free strict tensor category on a semi-tensor scheme.
\section{Planar order and POP-graph}

In \cite{[HLY16]}, we gave a combinatorial characterization of an equivalence classes of a processive plane graph in terms of a planar order on its underlying processive graph. In this paper, we define planar order for any poset.

\begin{defn}
 A \textbf{planar order} on a poset $(X,\rightarrow)$ is a linear order $\prec$ on $X$, such that

$(P_1)$ for any $x_1,x_2\in X$, $x_1\rightarrow x_2$ implies $x_1\prec x_2$;

$(P_2)$ for any $x_1,x_2,x_3\in X$,  $x_1\prec x_2\prec x_3$ and $x_1\rightarrow x_3$ imply that either $x_1\rightarrow x_2$ or $x_2\rightarrow x_3$.

\end{defn}
$(P_1)$ says that $\prec$ is a linear extension of $\rightarrow$.

Recall that two partial orders on a set are \textbf{conjugate} if each pair of elements are comparable by exactly one of them. It is easy to see that $(P_2)$ is equivalent to the condition that if $e_1\prec e_2 \prec e_3$, then $e_1\not\rightarrow e_2$ and $e_2\not\rightarrow e_3$ imply that $e_1\not\rightarrow e_3$. Thus $(P_2)$ enables us to define a transitive binary relation: $e_1<e_2$ if and only if $e_1\prec e_2$ and $e_1\not\rightarrow e_2$; moreover, if $(P_1)$ is satisfied, then the linearity of $\prec$ implies that $<$ is a conjugate order of $\rightarrow$. So the planar order $\prec$ is a reformulation of the conjugate order of $\rightarrow$.

In a directed graph, we denote $e_1\rightarrow e_2$ if there is a directed path starting from edge $e_1$ and ending with edge $e_2$. Similarly, $v_1\rightarrow v_2$ denotes that there is a directed path starting from vertex $v_1$ and ending with vertex $v_2$. For any acyclic directed graph, its edge set and vertex set are posets with the relation $e_1\rightarrow e_2$ and $v_1\rightarrow v_2$. We call them \textbf{edge poset} and \textbf{vertex poset} of the acyclic directed graph, respectively.

The following is a key notion in \cite{[HLY16]}.

\begin{defn}
A \textbf{planarly ordered processive graph} or \textbf{POP-graph} , is a processive graph $G$ together with a planar order $\prec$ on its edge poset $(E(G),\rightarrow)$.
\end{defn}
We simply denote  a POP-graph as $(G,\prec)$; see Fig $2$ for an example.
\begin{center}

\begin{tikzpicture}[scale=0.4]
\node (v2) at (-4,3) {};
\draw[fill] (-1.5,5.5) circle [radius=0.11];
\node (v1) at (-1.5,5.5) {};
\node (v7) at (-1.5,1) {};
\node (v9) at (1.5,5.5) {};
\node (v14) at (2,1.5) {};
\node (v3) at (-3,7.5) {};
\node (v4) at (-2,7.5) {};
\node (v5) at (-0.5,7.5) {};
\node (v6) at (-4.8,7.5) {};
\node (v11) at (-4.5,-1) {};
\node (v12) at (-2,-1) {};
\node (v13) at (0,-1) {};
\node (v15) at (2,-1) {};
\node (v8) at (1,7.5) {};
\node (v10) at (2.5,7.5) {};

\node [scale=0.8]  at (-2.5,3.5) {$6$};
\node[scale=0.8]  at (-3,5.2) {$5$};
\node[scale=0.8]  at (-1.2,3.3) {$9$};
\node[scale=0.8]  at (0.5,3.25) {$12$};
\node[scale=0.8]  at (2.2,3.7) {$15$};
\node[scale=0.8]  at (-3,1.7) {$8$};

\node[scale=0.8] [above] at (-3,7.5) {};
\node[scale=0.8] [above]  at (-2,7.5) {};
\node[scale=0.8] [above] at (-0.5,7.5) {};
\node [scale=0.8][above]  at (-4.8,7.5) {};
\node [scale=0.8][below] at (-4.5,-1) {};
\node [scale=0.8][below]at (-2,-1) {};
\node [scale=0.8][below] at (0,-1) {};
\node [scale=0.8][below] at (2,-1) {};
\node [scale=0.8][above]  at (1,7.5) {};
\node [scale=0.8][above]  at (2.5,7.5) {};

\node  at (-2.5,3.5) {};
\node  at (-3,5.2) {};
\node  at (-1.2,3.3) {};
\node  at (0.5,3.25) {};
\node  at (2.2,3.7) {};
\node  at (-3,1.7) {};
\draw[fill] (-4,3) circle [radius=0.11];
\draw[fill] (v1) circle [radius=0.11];
\draw[fill] (v7) circle [radius=0.11];
\draw[fill] (v9) circle [radius=0.11];
\draw[fill] (v14) circle [radius=0.11];
\draw[fill] (v1) circle [radius=0.11];
\draw[fill] (v2) circle [radius=0.11];
\draw[fill] (v3) circle [radius=0.11];
\draw[fill] (v4) circle [radius=0.11];
\draw[fill] (v5) circle [radius=0.11];
\draw[fill] (v6) circle [radius=0.11];
\draw[fill] (v8) circle [radius=0.11];
\draw[fill] (v10) circle [radius=0.11];
\draw[fill] (v11) circle [radius=0.11];
\draw[fill] (v12) circle [radius=0.11];
\draw[fill] (v13) circle [radius=0.11];
\draw[fill] (v15) circle [radius=0.11];

\draw  plot[smooth, tension=1] coordinates {(v1) (-2.5,5)  (-3.5,4) (v2)}[postaction={decorate, decoration={markings,mark=at position .5 with {\arrow[black]{stealth}}}}];
\draw  plot[smooth, tension=1] coordinates {(v1) (-2,4.5)  (-3,3.5) (v2)}[postaction={decorate, decoration={markings,mark=at position .5 with {\arrow[black]{stealth}}}}];

\draw  (-3,7.5) -- (-1.5,5.5)[postaction={decorate, decoration={markings,mark=at position .5 with {\arrow[black]{stealth}}}}];
\draw  (-2,7.5) -- (-1.5,5.5)[postaction={decorate, decoration={markings,mark=at position .5 with {\arrow[black]{stealth}}}}];
\draw  (-0.5,7.5) -- (-1.5,5.5)[postaction={decorate, decoration={markings,mark=at position .5 with {\arrow[black]{stealth}}}}];

\draw  (-4.8,7.5)-- (-4,3)[postaction={decorate, decoration={markings,mark=at position .5 with {\arrow[black]{stealth}}}}];
\draw  (-1.5,5.5)  -- (-1.5,1)[postaction={decorate, decoration={markings,mark=at position .5 with {\arrow[black]{stealth}}}}];
\draw  (-4,3) -- (-1.5,1)[postaction={decorate, decoration={markings,mark=at position .5 with {\arrow[black]{stealth}}}}];

\draw (1,7.5)--(1.5,5.5)[postaction={decorate, decoration={markings,mark=at position .5 with {\arrow[black]{stealth}}}}];
\draw  (2.5,7.5) -- (1.5,5.5)[postaction={decorate, decoration={markings,mark=at position .5 with {\arrow[black]{stealth}}}}];
\draw  (1.5,5.5) -- (-1.5,1)[postaction={decorate, decoration={markings,mark=at position .5 with {\arrow[black]{stealth}}}}];
\draw  (-4,3) -- (-4.5,-1)[postaction={decorate, decoration={markings,mark=at position .5 with {\arrow[black]{stealth}}}}];
\draw  (-1.5,1) -- (-2,-1)[postaction={decorate, decoration={markings,mark=at position .65 with {\arrow[black]{stealth}}}}];
\draw  (0,-1) -- (-1.5,1)[postaction={decorate, decoration={markings,mark=at position .5 with {\arrowreversed[black]{stealth}}}}];
\draw  (1.5,5.5) -- (2,1.5)[postaction={decorate, decoration={markings,mark=at position .5 with {\arrow[black]{stealth}}}}];
\draw  (2,1.5) -- (2,-1)[postaction={decorate, decoration={markings,mark=at position .5 with {\arrow[black]{stealth}}}}];

\node (v17) at (6.5,7.5) {};
\node [right][scale=0.8] at (6.5,3.5) {$19$};
\node (v16) at (6.5,-1) {};
\draw[fill] (v16) circle [radius=0.11];
\draw[fill] (v17) circle [radius=0.11];
\draw  (6.5,-1) -- (6.5,7.5)[postaction={decorate, decoration={markings,mark=at position .5 with {\arrowreversed[black]{stealth}}}}];
\node (v18) at (4.5,7.5) {};
\node [above][scale=0.8] at (4.5,7.5) {};
\node (v19) at (4.5,-1) {};
\node [below][scale=0.8] at (4.5,-1) {};
\node (v20) at (4.5,3.25) {};
\draw[fill] (v18) circle [radius=0.11];
\draw[fill] (v19) circle [radius=0.11];

\draw[fill] (v20) circle [radius=0.11];
\draw  (4.5,-1) -- (4.5,3.25)[postaction={decorate, decoration={markings,mark=at position .5 with {\arrowreversed[black]{stealth}}}}];
\draw  (4.5,3.24) -- (4.5,7.5)[postaction={decorate, decoration={markings,mark=at position .5 with {\arrowreversed[black]{stealth}}}}];

\node [scale=0.8]at (-5,6.5) {$1$};
\node [scale=0.8]at (-3,6.5) {$2$};
\node [scale=0.8]at (-1.5,7) {$3$};
\node [scale=0.8]at (-0.5,6.5) {$4$};
\node [scale=0.8]at (0.6,6.5) {$10$};
\node [scale=0.8]at (2.6,6.5) {$11$};
\node [scale=0.8]at (5,6) {$17$};
\node [scale=0.8]at (-4.7,0.5) {$7$};
\node [scale=0.8]at (-2.5,0) {$13$};
\node [scale=0.8]at (0,0) {$14$};
\node [scale=0.8]at (2.5,0) {$16$};
\node [scale=0.8]at (5,0.5) {$18$};
\end{tikzpicture}

Figure $2$. A POP-graph
\end{center}

A basic result is the following.
\begin{thm}[\cite{[HLY16]}]\label{M}
There is a bijection between POP-graphs and equivalence classes of processive plane graphs.
\end{thm}
The POP-graph in Fig $2$ corresponds to the processive plane graph in Fig $1$.

\section{Planar order on vertices}
In this section, we will prove our main result. Before that we need some preliminaries.

From now on, we fix be a POP-graph $(G,\prec)$. For a vertex $v$ of $(G,\prec)$,  the set $I(v)$ of incoming edges  and the set $O(v)$ of outgoing edges  are linearly ordered by $\prec$. We introduce some notations when $I(v)$ or $O(v)$ are not empty:
$$i^-(v)=min\ I(v),$$
$$i^+(v)=max\  I(v),$$
$$o^-(v)=min\ O(v),$$
$$o^+(v)=max\  O(v).$$

The following lemma is a result first proved in \cite{[LY16]}.
\begin{lem}\label{P_2 to U_2}
Let $v$ be a vertex of $(G,\prec)$. If the degree of $v$ is not one, then $o^-(v)=i^+(v)+1$ under the linear order $\prec$.
\end{lem}
\begin{proof}
Notice that $G$ is a processive graph, then $deg(v)\neq 1$ implies that $I(v)\neq\emptyset$ and $O(v)\neq\emptyset$. Thus both $i^+(v)$ and $o^-(v)$ exist. Now we prove $o^-(v)=i^+(v)+1$ by contradiction. Suppose there exists an edge $e$, such that $i^+(v)\prec e\prec o^-(v)$. Since $i^+(v)\rightarrow o^-(v)$, then by $(P_2)$ we have $i^+(v)\rightarrow e$ or $e\rightarrow o^-(v)$. If $i^+(v)\rightarrow e$, then there must exists an edge $e'\in O(v)-\{o^-(v)\}$, such that $e'\rightarrow e$ or $e'=e$. Thus $e'\preceq e$, which contradicts $e\prec o^-(v)$. Otherwise, $e\rightarrow o^-(v)$, then there must exist an edge $e''\in I(v)-\{i^+(v)\}$ such that $e\rightarrow e''$ or $e''=e$. Then $e\preceq e''$, which contradicts $i^+(v)\prec e$.
\end{proof}
Lemma \ref{P_2 to U_2} shows that for any vertex $v$, $\overline{E(v)}=\overline{I(v)}\sqcup \overline{O(v)}$, where $E(v)$ is the set of incident edges of $v$ and $\overline{X}$ denotes the interval of subset $X$ in a poset. Due to Lemma \ref{P_2 to U_2},  we can define a linear order $\prec_V$ on the vertex set $V(G)$.
For any two  different vertices $v_1, v_2$ of $G$,   $v_1\prec_V v_2$  if and only if one of the following conditions is satisfied:

$(1)$ $I^+(v_1)\prec I^+(v_2)$, $(2)$ $I^+(v_1)\prec O^-(v_2)$, $(3)$ $O^-(v_1) \preceq I^+(v_2)$, $(4)$ $O^-(v_1)\prec O^-(v_2)$.\\

\noindent We write $v_1\preceq_V v_2$ if $v_1=v_2$ or $v_1\prec_V v_2$. The following Theorem is our main result.
\begin{thm}\label{main}
For any POP-graph $(G,\prec)$, $\preceq_V$ defines a planar order on the vertex poset $(V(G),\rightarrow)$.
\end{thm}
\begin{proof}$(1)$ $\prec_V$ satisfies $(P_1)$.
If $v_1\rightarrow v_2$, then there exist $e_i\in E(G)$ $(1\leq i\leq n)$ such that $v_1=s(e_1)$, $v_2=t(e_n)$ and $t(e_i)=s(e_{i+1})$ for $(1\leq i\leq n-1)$, which implies that $o^-(v)\preceq e_1\preceq e_n\preceq i^+(v_2)$. Thus $o^-(v_1)\preceq i^+(v_2)$, then by definition of $\prec_V$, we have $v_1\prec_V v_2$.

$(2)$ $\prec_V $ satisfies $(P_2)$.
Suppose $v_1\prec_V v_2\prec_V v_3$ and $v_1\rightarrow v_3$, then $o^-(v_1)$ and $i^+(v_3)$ exist and $o^-(v_1)\prec i^+(v_3)$.
We have four cases:

\textbf{Case 1:} $v_1$ is a source and $v_3$ is a sink. In this case, notice that $G$ is processive, then by Definition $2$, $\{o^-(v_1)\}=O(v_1)$ and $\{i^+(v_3)\}=I(v_3)$. So $v_1\rightarrow v_3$ implies that $o^-(v_1)\rightarrow i^+(v_3)$. Let $e=i^+(v_2)$ or $o^-(v_2)$, then $v_1\prec_V v_2\prec_V v_3$ implies that $o^-(v_1)\prec e\prec i^+(v_3)$ or $o^-(v_1)= e$ or $e= i^+(v_3)$. In the first case, by $(P_2)$, we have $o^-(v_1)\rightarrow e$ or $ e\rightarrow i^+(v_3)$, which implies that $v_1\rightarrow v_2$ or $v_2\rightarrow v_3$. In the second case, we have $v_1\rightarrow v_2$, and in the third case, we have $v_2\rightarrow v_3$.

\textbf{Case 2:} $v_1$ is not a source and $v_3$ is a sink. In this case, $i^+(v_1)$ exists and by Definition $2$, $\{i^+(v_3)\}=I(v_3)$. So $v_1\rightarrow v_3$ implies that $i^+(v_1)\rightarrow i^+(v_3)$. Let $e=i^+(v_2)$ or $o^-(v_2)$, then $v_1\prec_V v_2\prec_V v_3$ implies that $i^+(v_1)\prec e\prec i^+(v_3)$ or  $e= i^+(v_3)$. In the first case, by $(P_2)$, we have $i^+(v_1)\rightarrow e$ or $ e\rightarrow i^+(v_3)$, which implies that $v_1\rightarrow v_2$ or $v_2\rightarrow v_3$. In the second case, we have $v_2\rightarrow v_3$.

\textbf{Case 3:} $v_1$ is  a source and $v_3$ is not a sink. This case is similar to case $(2)$.

\textbf{Case 4:} $v_1$ is not a source and $v_3$ is not a sink. In this case, both $i^+(v_1)$ and $o^-(v_2)$ exist and $v_1\rightarrow v_3$ implies that $i^+(v_1)\rightarrow o^-(v_3)$. Let $e=i^+(v_2)$ or $o^-(v_2)$, then $v_1\prec_V v_2\prec_V v_3$ implies that $i^+(v_1)\prec e\prec i^+(v_3)$. By $(P_2)$, we have $i^+(v_1)\rightarrow e$ or $ e\rightarrow o^-(v_3)$, which implies that $v_1\rightarrow v_2$ or $v_2\rightarrow v_3$.
\end{proof}
Figure $3$ shows the planar order on the vertex poset of the POP-graph in Fig $2$.
\begin{center}

\begin{tikzpicture}[scale=0.4]
\node (v2) at (-4,3) {};
\draw[fill] (-1.5,5.5) circle [radius=0.11];
\node (v1) at (-1.5,5.5) {};
\node (v7) at (-1.5,1) {};
\node (v9) at (1.5,5.5) {};
\node (v14) at (2,1.5) {};
\node (v3) at (-3,7.5) {};
\node (v4) at (-2,7.5) {};
\node (v5) at (-0.5,7.5) {};
\node (v6) at (-4.8,7.5) {};
\node (v11) at (-4.5,-1) {};
\node (v12) at (-2,-1) {};
\node (v13) at (0,-1) {};
\node (v15) at (2,-1) {};
\node (v8) at (1,7.5) {};
\node (v10) at (2.5,7.5) {};

\node [scale=0.8]  at (-2.5,3.5) {};
\node[scale=0.8]  at (-3,5.2) {};
\node[scale=0.8]  at (-1.2,3.3) {};
\node[scale=0.8]  at (0.5,3.25) {};
\node[scale=0.8]  at (2.2,3.7) {};
\node[scale=0.8]  at (-3,1.7) {};

\node[scale=0.8] [above] at (-3,7.5) {$2$};
\node[scale=0.8] [above]  at (-2,7.5) {$3$};
\node[scale=0.8] [above] at (-0.5,7.5) {$4$};
\node [scale=0.8][above]  at (-4.8,7.5) {$1$};
\node [scale=0.8][below] at (-4.5,-1) {$7$};
\node [scale=0.8][below]at (-2,-1) {$12$};
\node [scale=0.8][below] at (0,-1) {$13$};
\node [scale=0.8][below] at (2,-1) {$15$};
\node [scale=0.8][above]  at (1,7.5) {$8$};
\node [scale=0.8][above]  at (2.5,7.5) {$9$};

\node  at (-2.5,3.5) {};
\node  at (-3,5.2) {};
\node  at (-1.2,3.3) {};
\node  at (0.5,3.25) {};
\node  at (2.2,3.7) {};
\node  at (-3,1.7) {};
\draw[fill] (-4,3) circle [radius=0.11];
\draw[fill] (v1) circle [radius=0.11];
\draw[fill] (v7) circle [radius=0.11];
\draw[fill] (v9) circle [radius=0.11];
\draw[fill] (v14) circle [radius=0.11];
\draw[fill] (v1) circle [radius=0.11];
\draw[fill] (v2) circle [radius=0.11];
\draw[fill] (v3) circle [radius=0.11];
\draw[fill] (v4) circle [radius=0.11];
\draw[fill] (v5) circle [radius=0.11];
\draw[fill] (v6) circle [radius=0.11];
\draw[fill] (v8) circle [radius=0.11];
\draw[fill] (v10) circle [radius=0.11];
\draw[fill] (v11) circle [radius=0.11];
\draw[fill] (v12) circle [radius=0.11];
\draw[fill] (v13) circle [radius=0.11];
\draw[fill] (v15) circle [radius=0.11];

\draw  plot[smooth, tension=1] coordinates {(v1) (-2.5,5)  (-3.5,4) (v2)}[postaction={decorate, decoration={markings,mark=at position .5 with {\arrow[black]{stealth}}}}];
\draw  plot[smooth, tension=1] coordinates {(v1) (-2,4.5)  (-3,3.5) (v2)}[postaction={decorate, decoration={markings,mark=at position .5 with {\arrow[black]{stealth}}}}];

\draw  (-3,7.5) -- (-1.5,5.5)[postaction={decorate, decoration={markings,mark=at position .5 with {\arrow[black]{stealth}}}}];
\draw  (-2,7.5) -- (-1.5,5.5)[postaction={decorate, decoration={markings,mark=at position .5 with {\arrow[black]{stealth}}}}];
\draw  (-0.5,7.5) -- (-1.5,5.5)[postaction={decorate, decoration={markings,mark=at position .5 with {\arrow[black]{stealth}}}}];

\draw  (-4.8,7.5)-- (-4,3)[postaction={decorate, decoration={markings,mark=at position .5 with {\arrow[black]{stealth}}}}];
\draw  (-1.5,5.5)  -- (-1.5,1)[postaction={decorate, decoration={markings,mark=at position .5 with {\arrow[black]{stealth}}}}];
\draw  (-4,3) -- (-1.5,1)[postaction={decorate, decoration={markings,mark=at position .5 with {\arrow[black]{stealth}}}}];

\draw (1,7.5)--(1.5,5.5)[postaction={decorate, decoration={markings,mark=at position .5 with {\arrow[black]{stealth}}}}];
\draw  (2.5,7.5) -- (1.5,5.5)[postaction={decorate, decoration={markings,mark=at position .5 with {\arrow[black]{stealth}}}}];
\draw  (1.5,5.5) -- (-1.5,1)[postaction={decorate, decoration={markings,mark=at position .5 with {\arrow[black]{stealth}}}}];
\draw  (-4,3) -- (-4.5,-1)[postaction={decorate, decoration={markings,mark=at position .5 with {\arrow[black]{stealth}}}}];
\draw  (-1.5,1) -- (-2,-1)[postaction={decorate, decoration={markings,mark=at position .65 with {\arrow[black]{stealth}}}}];
\draw  (0,-1) -- (-1.5,1)[postaction={decorate, decoration={markings,mark=at position .5 with {\arrowreversed[black]{stealth}}}}];
\draw  (1.5,5.5) -- (2,1.5)[postaction={decorate, decoration={markings,mark=at position .5 with {\arrow[black]{stealth}}}}];
\draw  (2,1.5) -- (2,-1)[postaction={decorate, decoration={markings,mark=at position .5 with {\arrow[black]{stealth}}}}];

\node (v17) at (6.5,7.5) {};
\node [right][scale=0.8] at (6.5,3.5) {};
\node (v16) at (6.5,-1) {};
\draw[fill] (v16) circle [radius=0.11];
\draw[fill] (v17) circle [radius=0.11];
\draw  (6.5,-1) -- (6.5,7.5)[postaction={decorate, decoration={markings,mark=at position .5 with {\arrowreversed[black]{stealth}}}}];
\node (v18) at (4.5,7.5) {};
\node [above][scale=0.8] at (4.5,7.5) {$16$};
\node (v19) at (4.5,-1) {};
\node [below][scale=0.8] at (4.5,-1) {$18$};
\node (v20) at (4.5,3.25) {};
\draw[fill] (v18) circle [radius=0.11];
\draw[fill] (v19) circle [radius=0.11];

\draw[fill] (v20) circle [radius=0.11];
\draw  (4.5,-1) -- (4.5,3.25)[postaction={decorate, decoration={markings,mark=at position .5 with {\arrowreversed[black]{stealth}}}}];
\draw  (4.5,3.24) -- (4.5,7.5)[postaction={decorate, decoration={markings,mark=at position .5 with {\arrowreversed[black]{stealth}}}}];

\node [scale=0.8]at (-2.2,5.6) {$5$};
\node [scale=0.8]at (-4.7,3) {$6$};
\node [scale=0.8]at (-2.2,1) {$11$};
\node [scale=0.8]at (0.8,5.5) {$10$};
\node [scale=0.8]at (1.3,1.5) {$14$};
\node [scale=0.8]at (3.8,3.25) {$17$};
\node [scale=0.8, above]at (6.5,7.5) {$19$};
\node [scale=0.8, below]at (6.5,-1) {$20$};
\end{tikzpicture}

Figure $3$. Induced planar order on vertices.
\end{center}

Theorem \ref{main} shows that for any processive graph, each conjugate order of its edge poset induces a conjugate order of its vertex poset. However, in general, the converse is not true. Therefore, together with Theorem 2.1,  Theorem 3.1 demonstrates that edge poset is more effective tool than vertex poset in the study of upward planarity. It is worth to mention that  Fraysseix and Mendez, in a different but essentially equivalent  context,  also  showed a similar judgement (in their final remark of  \cite{[FM96]}). In our subsequent work, we will show that for a transitive reduced processive graph (directed covering graph of a poset), a planar order on its vertex set can naturally induce a planar order on its edge set, which is essentially related the work in  \cite{[FM96]}.

\section*{Acknowlegement}
This work is supported by "the Fundamental Research Funds for the Central Universities".

\par


\textbf{Xuexing Lu}\hfill \\  Email: xxlu@mail.ustc.edu.cn

\end{document}